\documentclass[11pt]{amsart}

\usepackage{enumerate,url,amssymb,mathrsfs}
\usepackage{graphicx}
\newtheorem{theorem}{Theorem}[section]
\newtheorem{lemma}[theorem]{Lemma}
\newtheorem*{lemma*}{Lemma}
\newtheorem{proposition}[theorem]{Proposition}
\newtheorem{corollary}[theorem]{Corollary}

\theoremstyle{definition}

\newtheorem{questionstar}[theorem]{Question}
\newtheorem*{question}{Question}
\newtheorem{problem}[theorem]{Problem}

\theoremstyle{remark}
\newtheorem{remark}[theorem]{Remark}

\numberwithin{equation}{section}

\newcommand{\abs}[1]{\lvert#1\rvert}
\newcommand{\norm}[1]{\lVert#1\rVert}

\newcommand{\C}{\mathbb{C}}
\newcommand{\CC}{\mathscr{C}}

\newcommand{\X}{\mathbb{X}}
\newcommand{\Y}{\mathbb{Y}}
\newcommand{\E}{\mathcal{E}}
\newcommand{\EE}{\mathsf{E}}

\newcommand{\R}{\mathbb{R}}

\newcommand{\W}{\mathscr{W}}
\newcommand{\V}{\mathcal{V}}

\newcommand{\const}{\mathrm{const}}

\newcommand{\onto}{\overset{{}_{\textnormal{\tiny{onto}}}}{\longrightarrow}}
\DeclareMathOperator{\diam}{diam}
\DeclareMathOperator{\dist}{dist}

\DeclareMathOperator{\re}{Re}
\DeclareMathOperator{\im}{Im}
\DeclareMathOperator{\Div}{div}

\DeclareMathOperator{\loc}{loc}

\def\XXint#1#2#3{{\setbox0=\hbox{$#1{#2#3}{\int}$}
\vcenter{\hbox{$#2#3$}}\kern-.5\wd0}}

\def\le{\leqslant}
\def\ge{\geqslant}
\begin{document}

\title{Diffeomorphic approximation of Sobolev~homeomorphisms}

\author{Tadeusz Iwaniec}
\address{Department of Mathematics, Syracuse University, Syracuse, NY 13244, USA
and Department of Mathematics and Statistics, University of Helsinki, Finland}
\email{tiwaniec@syr.edu}
\thanks{}

\author{Leonid V. Kovalev}
\address{Department of Mathematics, Syracuse University, Syracuse,
NY 13244, USA}
\email{lvkovale@syr.edu}
\thanks{}

\author{Jani Onninen}
\address{Department of Mathematics, Syracuse University, Syracuse,
NY 13244, USA}
\email{jkonnine@syr.edu}
\thanks{Iwaniec was supported by the NSF grant DMS-0800416 and the Academy of Finland grant 1128331.
Kovalev was supported by the NSF grant DMS-0968756.
Onninen was supported by the NSF grant  DMS-1001620.}




\begin{abstract}
Every homeomorphism $h \colon \mathbb X \to \mathbb Y$ between planar open sets that belongs to the Sobolev class
$\W^{1,p} (\mathbb X , \mathbb Y)$, $1<p<\infty$, can be approximated in the Sobolev norm by $\CC^\infty$-smooth diffeomorphisms.
\end{abstract}
\maketitle

\section{Introduction}
By the very definition, the Sobolev space $\W^{1,p}(\X ,\R)$, $1 \le p < \infty$, in a domain $\X \subset \R^n$, is the completion of $\CC^\infty$-smooth real functions having finite Sobolev norm
\[\norm u_{\W^{1,p}(\X)} = \norm{u}_{\mathscr L^p (\X)} + \norm{\nabla u}_{\mathscr L^p (\X)}<\infty.  \]
The question of smooth approximation becomes more intricate for Sobolev mappings, whose target is not a linear space, say a smooth manifold~\cite{Be, HIMO, HL1, HL2} or even for mappings between open subsets $\X, \Y$ of the Euclidean space $\R^n$. If a given homeomorphism $h \colon \X \onto \Y$ is in the Sobolev class $\W^{1,p}(\X, \Y)$ it is not obvious at all as to whether one can preserve injectivity property of the $\CC^\infty$-smooth approximating mappings. It is rather surprising that this question remained unanswered after the global invertibility of Sobolev mappings became an issue in nonlinear elasticity~\cite{Ba0, FG, MST, Sv}.  It was formulated and promoted by John M. Ball in the following form.

\begin{question}\cite{Ba, Ba2}
If $h \in \W^{1,p}(\X, \R^n)$ is invertible, can $h$ be approximated in $ \W^{1,p}$ by piecewise affine invertible mappings?
\end{question}

J. Ball attributes this question to L.C. Evans and points out its relevance to the  regularity of minimizers of neohookean energy functionals~\cite{Ba1, BPO1, CL, Ev, SiSp}. Partial results toward the Ball-Evans problem were obtained in~\cite{Mo} (for planar bi-Sobolev mappings that are smooth outside of a finite set) and in~\cite{BM} (for planar bi-H\"older mappings, with approximation in the H\"older norm).
The articles~\cite{Ba,SS} illustrate the difficulty of preserving invertibility in the  approximation process.
In~\cite{IKOhopf} we provided an affirmative answer to the Ball-Evans question in the planar case when $p=2$.
In the present paper we extend the result of~\cite{IKOhopf} to all Sobolev classes $\W^{1,p}(\X , \Y)$ with $1<p< \infty$. The case $p=1$ still remains open.

Let $\X$ be a nonempty open set in $\R^2$. We study complex-valued functions $h=u+iv \colon \X \to \C \simeq \R^2$ of Sobolev class $\W^{1,p} (\X, \C)$, $1<p<\infty$. Their real and imaginary part have well defined gradient in $\mathscr L^p (\X, \R^2)$
\[\nabla u \colon \X \to \R^2 \quad \mbox{ and } \quad \nabla v \colon \X \to \R^2.\]
Then we introduce the gradient mapping of $h$, by setting
\begin{equation}\label{grad}
\nabla h = (\nabla u , \nabla v) \colon \X \to \R^2 \times \R^2.
\end{equation}
The $\mathscr L^p$-norm of the gradient mapping  and the $p$-energy
 of $h$ are defined by
 \begin{equation}\label{penergy}
 \norm{\nabla h}_{\mathscr L^p (\X)} = \left[ \int_{\X} \left(\abs{\nabla u}^p + \abs{\nabla v}^p\right) \right]^\frac{1}{p}, \quad \EE_{\X}[h]=\EE_{\X}^p[h]= \norm{\nabla h}^p_{\mathscr L^p (\X)}.
 \end{equation}
The reader may wish to notice that this norm is slightly different from what can be found in other texts in which the authors use the differential matrix of $h$ instead of the gradient mapping, so
\begin{equation}\label{penergy2}
\norm{Dh}_{\mathscr L^p(\X)} = \left[ \int_{\X} \left(\abs{\nabla u}^2 + \abs{\nabla v}^2\right)^\frac{p}{2} \right]^\frac{1}{p}.
\end{equation}
Thus our approach involves \emph{coordinate-wise} $p$-harmonic mappings, which we still call $p$-harmonic for the sake of brevity.
We shall take an advantage of the gradient mapping on numerous occasions, by exploring the associated {\it uncoupled} system of real $p$-harmonic equations for mappings with smallest $p$-energy. Our  theorem reads as follows.

\begin{theorem}\label{thmmain}
Let $h \colon \X \onto \Y$ be an orientation-preserving homeomorphism in the Sobolev space $\W^{1,p}_{\loc}(\X, \Y)$,
$1<p<\infty$, defined for open sets $\X, \Y \subset \R^2$. Then there exist $\CC^\infty$-diffeomorphisms $h_\ell  \colon \X
\onto \Y$, $\ell=1,2,\dots$ such that
\begin{enumerate}[(i)]
\item $h_\ell -h\in \W^{1,p}_\circ (\X, \R^2)$, $\ell=1,2, \dots$\vskip0.15cm
\item $\lim\limits_{\ell \to \infty} (h_\ell -h)=0$, uniformly on $\X$\vskip0.15cm
\item $\lim\limits_{\ell \to \infty} \norm{\nabla h_\ell -\nabla h}_{\mathscr L^p (\X)}= 0$\vskip0.15cm
\item $\norm{\nabla h_\ell}_{\mathscr L^p (\X)} \le \norm{\nabla h}_{\mathscr L^p (\X)}$, for $\ell =1,2, \dots$ \vskip0.15cm
\item If $h$ is a $\CC^\infty$-diffeomorphism outside of a compact subset of $\X$, then there is a compact subset of $\X$ outside which  $h_\ell \equiv h$, for all $\ell = 1,2, \dots$
\end{enumerate}
\end{theorem}

A straightforward triangulation argument yields the following corollary.

\begin{corollary}
Let $h \colon \X \onto \Y$ be an orientation-preserving homeomorphism in the Sobolev space $\W^{1,p}_{\loc}(\X, \Y)$,
$1<p<\infty$, defined for open sets $\X, \Y \subset \R^2$. Then there exist piecewise affine homeomorphisms $h_\ell  \colon \X
\onto \Y$, $\ell=1,2,\dots$ such that
\begin{enumerate}[(i)]
\item $h_\ell -h\in \W^{1,p}_\circ (\X, \R^2)$, $\ell=1,2, \dots$\vskip0.15cm
\item $\lim\limits_{\ell \to \infty} (h_\ell -h)=0$, uniformly on $\X$\vskip0.15cm
\item $\lim\limits_{\ell \to \infty} \norm{\nabla h_\ell -\nabla h}_{\mathscr L^p (\X)}= 0$.\vskip0.15cm
\item If $h$ is affine outside of a compact subset of $\X$, then there is a compact subset of $\X$ outside which  $h_\ell \equiv h$, for all $\ell = 1,2, \dots$
\end{enumerate}
\end{corollary}

We conclude this introduction with a sketch of the proof. The construction of an approximating diffeomorphism involves five consecutive modifications of
$h$. Steps 1, 2, and 4 are $p$-harmonic replacements based on the Alessandrini-Sigalotti extension~\cite{AS} of  the Rad\'o-Kneser-Choquet Theorem. The other steps involve an explicit smoothing procedure along crosscuts.
For this, we adopted some lines of arguments used in J. Munkres' work~\cite{Mu}.

\section{$p$-harmonic mappings and preliminaries}\label{prel}
Let  $\Omega$  be a bounded domain in the complex plain $\C \simeq \R^2$. A function $u \colon \Omega \to \R$ in the Sobolev class $\W^{1,p}_{\loc}(\Omega)$, $1<p<\infty$, is called $p$-harmonic if
\begin{equation}
\Div\, \abs{\nabla u}^{p-2} \nabla u =0
\end{equation}
meaning that
\begin{equation}
\int_\Omega \langle \abs{\nabla u}^{p-2} \nabla u , \nabla \varphi \rangle =0 \qquad
\text{for every $\varphi \in \mathscr C_\circ^\infty (\Omega)$.}
\end{equation}
The first observation is that the gradient map $f= \nabla u \colon \Omega \to \R^2$ is  $K$-quasiregular with $1\le K \le \max \{p-1, 1/(p-1)\}$, see~\cite{BI}. Consequently $u\in \mathscr C^{1, \alpha}_{\loc}(\Omega)$ with some $0< \alpha = \alpha(p) \le 1$. In fact~\cite{IM} the foremost regularity of a $p$-harmonic function ($p\ne 2$) is $\mathscr C^{k, \alpha}_{\loc}(\Omega)$, where the integer $k \ge 1$ and the H\"older exponent $\alpha \in (0,1]$ are determined by the equation
\[k+\alpha = \frac{7p-6+\sqrt{p^2+12p-12}}{6p-6}> 1 + \frac{1}{3}.\]
Thus, regardless of the exponent $p$, we have $u \in \mathscr C^{1, \alpha}_{\loc}(\Omega)$
with $\alpha=1/3$. Clearly, by elliptic regularity theory, outside the singular set
\[\mathcal S= \big\{ z\in \Omega \colon \nabla u (z)=0  \big\}, \]
we have $u\in \mathscr C^\infty (\Omega \setminus \mathcal S)$. The singular set, being the set of zeros of a quasiregular mapping, consists of isolated points; unless $u \equiv \const$. Pertaining to regularity up to  the boundary, we consider a domain $\Omega$ whose boundary near a point $z_\circ \in \partial \Omega$ is a $\mathscr C^\infty$-smooth arc, say $\Gamma \subset \partial \Omega$. Precisely, we assume that there exist a disk $D=D(z_\circ , \epsilon )$ and a $\CC^\infty$-smooth diffeomorphism $\varphi \colon D \onto \C$ such that
\[
\begin{split}
\varphi (D \cap \Omega) &= \C_+= \{z \colon \im z >0\}\\
\varphi (\Gamma) &= \R= \{z \colon \im z =0\} \\
\varphi (D \setminus \overline{\Omega}) &= \C_-= \{z \colon \im z <0\}.
\end{split}
\]
\begin{proposition}[Boundary Regularity] Suppose $u \in \W^{1,p} (\Omega) \cap \CC (\overline{\Omega})$ is $p$-harmonic in $\Omega$ and $\CC^\infty$-smooth when restricted to $\Gamma$.
Then $u$ is $\CC^{1,\alpha}$-regular up to $\Gamma$, meaning that $u$ extends to $D$ as a $\CC^{1, \alpha}(D)$-regular function, where $\alpha$ depends only on $p$.
\end{proposition}

\subsection{The Dirichlet problem}
There are two  formulations of the Dirichlet boundary value problem for $p$-harmonic equation; both are essential for our investigation. We begin with the variational formulation.

\begin{lemma}
Let $u_\circ \in \W^{1,p}(\Omega)$ be a given Dirichlet data. There exists precisely one function $u\in u_\circ +  \W_\circ^{1,p}(\Omega)$ which minimizes the $p$-harmonic energy:
\[\E_p[u]=\inf \left\{ \int_\Omega \abs{\nabla w}^p\colon w\in u_\circ +  \W_\circ^{1,p}(\Omega) \right\}.\]
\end{lemma}

The solution $u$ is certainly a $p$-harmonic function, so $\CC^{1, \alpha}_{\loc} (\Omega)$-regular. However, more efficient to us will be the following classical formulation of the Dirichlet problem.

\begin{problem} Given $u_\circ \in \CC(\partial {\Omega})$ find  a $p$-harmonic function $u$ in $\Omega$ which extends continuously to $\overline{\Omega}$ such that $u_{|_{\partial \Omega}}=u_\circ$.
\end{problem}

It is not difficult to see that such solution (if exists) is unique. However, the existence poses rather delicate conditions on  $\partial \Omega$ and the data $u_\circ \in \CC (\overline{\Omega})$.  We shall confine ourselves to Jordan  domains $\Omega \subset \C$ and the Dirichlet data $u_\circ \in \CC (\overline{\Omega})$ of finite $p$-harmonic energy.
In this case both formulations are valid and lead to the same solution. Indeed, the variational solution is continuous up to the boundary because each boundary point of a planar Jordan domain is a regular point for
the $p$-Laplace operator $\Delta_p$~\cite[p.418]{Ha}. See~\cite[6.16]{HKMb} for the discussion of boundary regularity and relevant capacities and~\cite[Lemma 2]{Le} for a capacity estimate that applies to simply connected domains.

\begin{proposition}[Existence]\label{proexist} Let $\Omega \subset \C$ be a bounded Jordan domain and $u_\circ \in \W^{1,p}(\Omega) \cap \CC (\overline{\Omega})$. There exists, unique, $p$-harmonic function $u\in \W^{1,p}(\Omega) \cap \CC (\overline{\Omega})$ such that $u_{|_{\partial \Omega}}=u_{\circ |_{\partial \Omega}}$.
\end{proposition}

\subsection{Rad\'o-Kneser-Choquet Theorem}
Let $h=u+iv$ be a complex harmonic mapping in a Jordan domain $\mathbb U$ that is continuous on $\overline{\mathbb U}$. Assume that the boundary mapping $h \colon \partial \mathbb U \onto \Gamma$ is an orientation-preserving homeomorphism onto a convex Jordan curve. Then $h$ is a $\CC^\infty$-smooth diffeomorphism of $\mathbb U$ onto the bounded component of $\C \setminus \Gamma$. Thus, in particular, the Jacobian determinant $J(z,h)= \abs{h_z}^2 - \abs{h_{\bar z}}^2$ is strictly positive in $\mathbb U$, see~\cite[p.20]{Dub}. Suppose, in addition, that $\partial \mathbb U$ contains a $\CC^\infty$-smooth arc $\gamma \subset \partial \mathbb U$, and $h$ takes $\gamma$ onto a $\CC^\infty$-smooth subarc in $\Gamma$. Then $h$ is $\CC^\infty$-smooth up to $\gamma$ and its Jacobian determinant is positive on $\gamma$ as well, see~\cite[p.116]{Dub}. Numerous presentations of the proof of Rad\'o-Kneser-Choquet Theorem can be found,~\cite{Dub}. The idea that goes back to Kneser~\cite{Kn} and Choquet~\cite{Ch} is to look at the structure of  the level curves of the coordinate functions $u=\re h$, $v=\im h$ and their linear combinations.
These ideas have been applied to more general linear and nonlinear elliptic systems of PDEs in the complex plane~\cite{BMN}, see also \cite{A2, AN, Lew, Ma} for related problems concerning critical points. In the present paper we shall explore a result due to G. Alessandrini and M. Sigalotti~\cite{AS} for a nonlinear system that consists of two $p$-harmonic equations
\[\begin{cases}
\Div \abs{\nabla u}^{p-2}\nabla u =0\\
\Div \abs{\nabla v}^{p-2}\nabla v =0
\end{cases}, \qquad 1<p< \infty, \quad h=u+iv.\]
Call it \emph{uncoupled $p$-harmonic system}. The novelty and key element in~\cite{AS} is the associated single linear elliptic PDE of divergence type (with variable coefficients) for a linear combination of $u$ and $v$. Such combination represents a real part of a quasiregular mapping and, therefore, admits only isolated critical points. We shall not go into their arguments in detail, but instead extract the following $p$-harmonic analogue of the Rad\'o-Kneser-Choquet Theorem.
\begin{theorem}[G. Alessandrini and M. Sigalotti]\label{RKCthm} Let $\mathbb U$ be a bounded Jordan domain and $h=u+iv \colon \overline{\mathbb U} \to \C$ be a continuous mapping whose coordinate functions $u,v \in \W^{1,p}(\mathbb U)$, $1<p<\infty$, are $p$-harmonic. Suppose that $h \colon \partial \mathbb U \onto \gamma$ is an orientation-preserving homeomorphism onto a convex Jordan curve $\gamma$. Then
\begin{enumerate}[(i)]
\item $h$ is a $\CC^\infty$-diffeomorphism from $\mathbb U$ onto the bounded component of $\C \setminus \gamma$. In particular,
\[J(z,h)= \abs{h_z}^2-\abs{h_{\bar z}}^2 >0 \quad \mbox{ in } \mathbb U.\]
\item If, in addition, $\partial \mathbb U$ contains a $\CC^\infty$-smooth arc $\Gamma \subset \partial \mathbb U$ and $h(\Gamma)$ is a $\CC^\infty$-smooth subarc in $\gamma$, then $h$ is $\CC^{1, \alpha}$-regular up to $\Gamma$, for some $0< \alpha = \alpha (p) <1$ (actually $\CC^\infty$). Moreover $J(z,h)>0$ on $\Gamma$ as well.
\end{enumerate}
\end{theorem}
This theorem is a straightforward corollary of Theorem 5.1 in \cite{AS}. However, three remarks are in order.
\begin{enumerate}
\item In their Theorem 5.1 the authors of~\cite{AS} assume that $\mathbb U$ satisfies an exterior cone condition. This is needed only insofar as to ensure the existence of a continuous extension of a given homeomorphism $\Phi \colon \partial \mathbb U \to \gamma$ into $\mathbb U$ whose coordinate functions are $p$-harmonic in $\mathbb U$. Obviously, such an extension is unique, though the $p$-harmonic energy need not be finite. Once we have such a mapping the exterior cone condition on $\mathbb U$ for the conclusion of Theorem~5.1 is redundant, see Remark~3.2 in~\cite{AS}. This is exactly the case we are dealing with in Theorem~\ref{RKCthm}.
\item In regard to the statement (ii) we point out that in Theorem~5.1 of~\cite{AS} the authors work with the mappings that are smooth up to the entire boundary of $\mathbb U$. Nonetheless their proof that $J(z,h)>0$ on $\partial \mathbb U$ is local, so applies without any change to our case (ii).
\item Since $J(z,h) > 0$ in $\mathbb U$ up to the arc $\Gamma \subset \partial \mathbb U$ the coordinate functions of $h$ have nonvanishing gradient. This means that $p$-harmonic equation is uniformly elliptic up to $\Gamma$. Consequently, $h$ is $\CC^\infty$-smooth on $\mathbb U$ up to $\Gamma$.
\end{enumerate}

\subsection{The $p$-harmonic replacement}
Let $\Omega$ be a bounded domain in $\R^2 \simeq \C$. We consider a class $\mathcal A(\Omega)=\mathcal A^p(\Omega)$, $1<p< \infty$, of uniformly continuous functions $h=u+iv \colon \Omega \to \C$
 having finite $p$-harmonic energy and furnish it with the norm
\[\norm{h}_{\mathcal A^p(\Omega)} = \norm{h}_{\CC(\Omega)} + \norm{\nabla h}_{\mathscr L^p (\Omega)}.\]
The closure of $\CC_\circ ^\infty (\Omega)$ in $\mathcal A^p (\Omega)$ will be denoted
by $\mathcal A_\circ^p (\Omega)$.

\begin{proposition}\label{proreplace}
Let $\mathbb U \Subset \Omega$ be a Jordan subdomain of $\Omega$. There exists a unique operator
\[\mathbf R_{\mathbb U} \colon \mathcal A^p (\Omega) \to \mathcal A^p (\Omega) \]
(nonlinear if $p\ne 2$) such that for every $h\in \mathcal A^p (\Omega)$
\begin{equation}\label{propcond1}
\begin{split}
\mathbf R_{\mathbb U} h &=h \qquad \mbox{in } \Omega \setminus \mathbb U\\
\mathbf R_{\mathbb U} &\in h+ \W^{1,p}_\circ (\mathbb U)\\
\Delta_p \mathbf R_{\mathbb U} h&=0 \qquad \mbox{in } \mathbb U
\end{split}
\end{equation}
\begin{equation}\label{propcond2}
\EE_\Omega [\mathbf R_{\mathbb U} h] \le \EE_\Omega [ h]
\end{equation}
Equality occurs in~\eqref{propcond2} if and only if $h$ is $p$-harmonic in $\mathbb U$.
\end{proposition}

\begin{proof}
For $h=u+iv$ we define
\[ \mathbf R_{\mathbb U} h = \mathbf R_{\mathbb U} u + i\,  \mathbf R_{\mathbb U} v.\]
It is therefore enough to construct the replacement for real-valued functions. For $u\in \mathcal A^p (\Omega)$ real,  we define
\[\mathbf R_{\mathbb U} u = \begin{cases} u \quad \mbox{in } \Omega \setminus \mathbb U \\
\tilde{u}  \quad \mbox{in } \mathbb U
\end{cases}\]
where $\tilde u$ is determined uniquely as a solution to the Dirichlet problem
\[\begin{cases} \Div \,\abs{\nabla \tilde{u}}^{p-2} \nabla \tilde{u}= 0 \quad \mbox{in } \mathbb U\\
\tilde{u}\in u + \W^{1,p}_\circ (\mathbb U)
\end{cases}
\]
so conditions~\eqref{propcond1} are fulfilled. That $\mathbf R_{\mathbb U} u$ is continuous in $\Omega$ is guaranteed by Proposition~\ref{proexist}. The solution $\tilde{u}$ is found as the minimizer of the $p$-harmonic energy in the class $u+\W^{1,p}_\circ (\mathbb U)$, so we certainly have
\[\EE_\Omega [\mathbf R_{\mathbb U} u] \le \EE_\Omega [ u] \]
The same estimate holds for the imaginary part of $h$, so adding them  up yields
\[\EE_\Omega [\mathbf R_{\mathbb U} h] \le \EE_\Omega [ h]. \qedhere\]
\end{proof}

\begin{remark}
The reader may wish to know that the operator $\mathbf R_{\mathbb U}\colon \mathcal A(\Omega)\to\mathcal A(\Omega)$ is continuous, though we do not appeal to this fact.
\end{remark}

\subsection{Smoothing along a crosscut}

 Consider a bounded Jordan domain $\mathbb U$ and a $\CC^\infty$-smooth crosscut $\Gamma \subset \mathbb U$ with two distinct end-points in $\partial \mathbb U$. By definition, this means that there is a $\CC^\infty$-diffeomorphism $\varphi \colon \C \onto \mathbb U$ such that $\Gamma= \varphi (\R)$, and its distinct endpoints are given by
\[
\begin{split}
\lim\limits_{x \to - \infty} \varphi (x) \in \partial \mathbb U \\
\lim\limits_{x \to  \infty} \varphi (x) \in \partial \mathbb U
\end{split}
\]
Such $\Gamma$ splits $\mathbb U$ into two Jordan subdomains
\[
\begin{split}
\mathbb U_+ = \varphi (\C_+), \quad \C_+= \{z\colon \im z >0\}\\
\mathbb U_- = \varphi (\C_-), \quad \C_-= \{z\colon \im z <0\}.
\end{split}
\]
Suppose we are given a homeomorphism $f \colon \overline{\mathbb U} \to \mathbb C$ such that each of two mappings
\[f \colon \mathbb U_+ \to \mathbb R^2  \quad \mbox{and} \quad  f \colon \mathbb U_- \to \mathbb R^2\]
is $\CC^\infty$-smooth up to $\Gamma$. Assume that for some constant $0<m<\infty$ we have
\[\abs{Df(z)}\le m \quad \mbox{ and } \quad \det Df(z) \ge \frac{1}{m}\]
 on $\mathbb U_+$ and on $\mathbb U_-$. Thus $f \colon \mathbb U \to \R^2$ is in fact locally bi-Lipschitz.

\begin{proposition}\label{prop41}
Under the above conditions there is a constant $0<M< \infty$ such that for every open set $\mathbb V \subset \mathbb U$ containing $\Gamma$ one can find a homeomorphism $g \colon \overline{\mathbb U} \onto f(\overline{\mathbb U})$ which is a $\CC^\infty$-diffeomorphism in $\mathbb U$, with the following properties:
\begin{equation}
g(z)=f(z), \mbox{ for } z\in (\overline{\mathbb U} \setminus \mathbb V) \cup \Gamma
\end{equation}
\begin{equation}
\abs{Dg(z)} \le M \quad \mbox{and} \quad \det Dg(z) > \frac{1}{M} \mbox{ on } \mathbb U.
\end{equation}
\end{proposition}
The key element of this smoothing device is that the constant $M$ is independent of the neighborhood $\mathbb V$ of $\Gamma$, see Figure~\ref{skull}. The proof is given in~\cite{IKOhopf} following the ideas of~\cite{Mu}.

\begin{center}
\begin{figure}[h]
\includegraphics[width=0.4\textwidth]{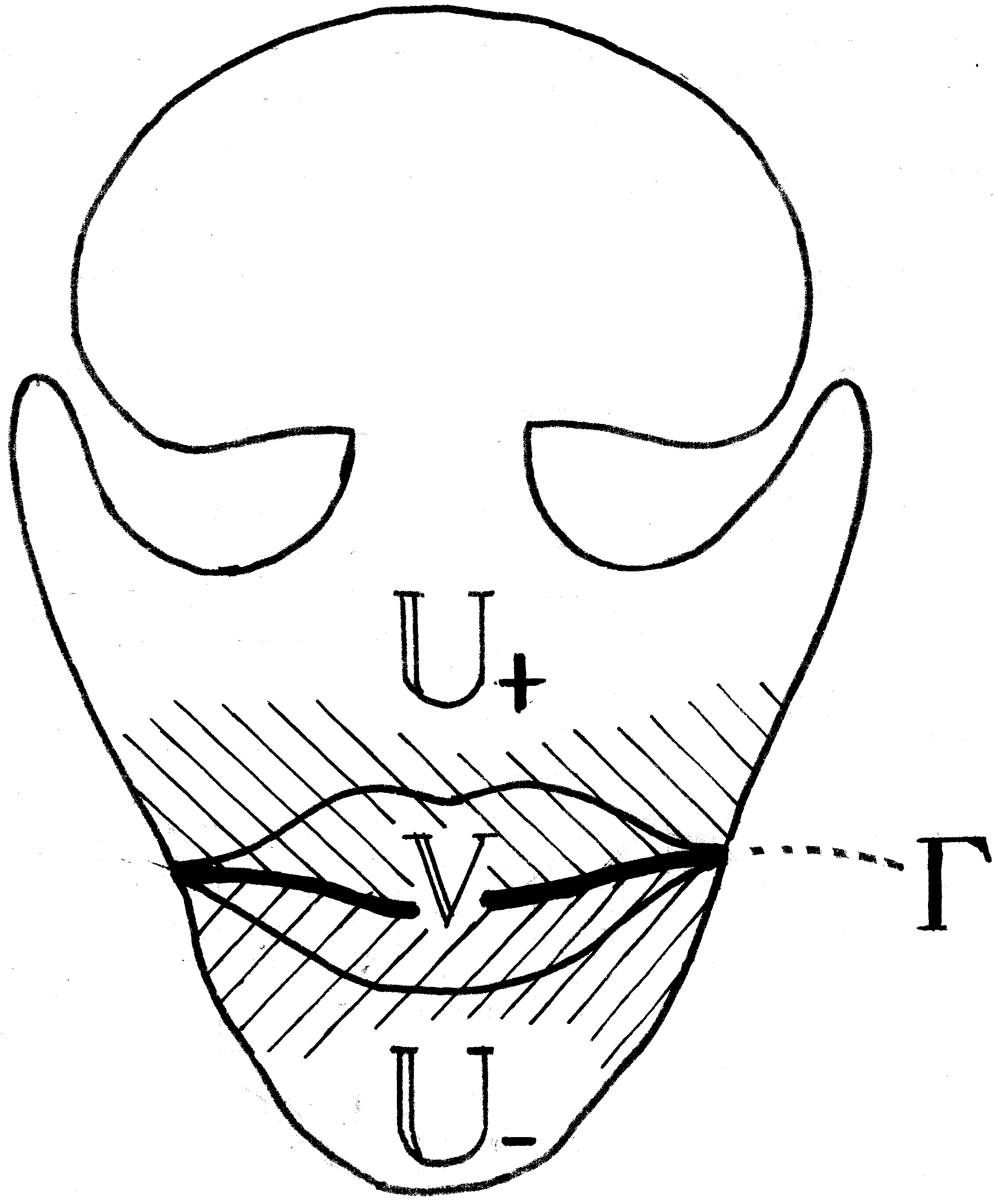}
\caption{Jordan domain with a crosscut $\Gamma$ and its neighborhood $\mathbb V$.}\label{skull}
\end{figure}
\end{center}

We shall recall similar smoothing device for cuts along Jordan curves. Let $\mathbb U$ be a simply connected domain with $\CC^\infty$-regular cut along a Jordan curve $\Gamma \subset \mathbb U$. This means there is a diffeomorphism $\varphi \colon \C \onto \mathbb U$ such that $\Gamma = \varphi (\mathbb S^1)$, $\mathbb S^1 =\{z \in \C \colon \abs{z}=1\}$. As before $\Gamma$ splits $\mathbb U$ into
\[
\begin{split}
\mathbb U_+ = \varphi (\mathbb D_+), \quad \mathbb D_+= \{z\colon \abs{z}<1\}\\
\mathbb U_- = \varphi (\mathbb D_-), \quad \mathbb D_-= \{z\colon \abs{z}>1\}.
\end{split}
\]
 Suppose we are given a homeomorphism $f \colon \mathbb U \to \mathbb R^2$ such that each of two mappings
\[
f \colon \mathbb U_+ \to \mathbb R^2  \quad \mbox{ and  } \quad  f \colon \mathbb U_- \to \mathbb R^2
\]
is $\CC^\infty$-smooth up to $\Gamma$. Assume that  for some constant $0<m<\infty$ we have
\[\abs{Df(z)}\le m \quad \mbox{ and } \quad \det Df(z) \ge \frac{1}{m}\]
 on $\mathbb U_+$ and $\mathbb U_-$.

\begin{proposition}\label{prop42}
Under the above conditions there is a constant $0<M< \infty$ such that for every open set $\mathbb V \subset \mathbb U$ containing $\Gamma$ one can find  a $\CC^\infty$-diffeomorphism $g \colon \mathbb U \onto f(\mathbb U)$ with the following properties
\begin{equation}
g(z)=f(z), \mbox{ for } z\in (\mathbb U \setminus \mathbb V) \cup \Gamma
\end{equation}
\begin{equation}
\abs{Dg(z)} \le M \quad \mbox{and} \quad \det Dg(z) > \frac{1}{M} \mbox{ on } \mathbb U.
\end{equation}
\end{proposition}

Having disposed of the above preliminaries we shall now proceed to the construction of the approximating sequence of diffeomorphisms.

\section{The proof}

\subsection{Scheme of the proof.}

Let us begin with a convention. We will often suppress the explicit dependence on  the Sobolev exponent $1<p< \infty$ in the notation, whenever it becomes selfexplanatory. For every $\epsilon >0$ we shall construct a $\CC^\infty$-diffeomorphism $\hslash \colon \X \onto \Y$ such that
\begin{enumerate}[(A)]
\item $\hslash -h \in \mathcal A_\circ (\X) $ \vspace{0.1cm}
\item $\norm{\hslash -h}_{\CC(\X)} \le \epsilon$ \vspace{0.1cm}
\item $\norm{\nabla \hslash -\nabla h}_{\mathscr L^p (\X)} \le \epsilon$ \vspace{0.1cm}
\item $\EE_{\X} [\hslash] \le \EE_{\X} [h]$ \vspace{0.1cm}
\item If $h$ is  a $\CC^\infty$-diffeomorphism outside of a compact subset of $\X$, then there exist a compact subset of $\X$ outside of which we have  $\hslash \equiv h$, for all $\epsilon >0$.
\end{enumerate}

We may and do assume that $h$ is not a $\CC^\infty$-diffeomorphism, since otherwise $\hslash =h$ satisfies the desired properties. Let $x_\circ \in \X$ be a point such that $h$ fails to be $\CC^\infty$-diffeomorphism in any neighborhood of $x_\circ$.

 We shall consider dyadic squares in $\Y$ with respect to a selected rectangular coordinate system in $\R^2$. By choosing the origin of the system we ensure that $h(x_\circ)$ does not lie on the boundary of any dyadic square.

 Let us  fix $\epsilon >0$. The construction of $\hslash$ proceeds in $5$ steps, each of which gives a homeomorphism $\hslash_k \colon \X \onto \Y$, $k=0,1, \dots , 5$, in the Sobolev class $\W^{1,p}_{\loc} (\X, \Y)$ such that $\hslash_0=h$, $\hslash_k \in \hslash_{k-1} +  \mathcal A_\circ (\X) $, $k=1,\dots, 5$ and $\hslash_5=\hslash$ is the desired diffeomorphism. For each $k=1,2, \dots , 5$ we will secure conditions analogous to (A)-(E). Namely,
\begin{enumerate}[($A_k$)]
\item $\hslash_k -\hslash_{k-1} \in\mathcal  A_\circ (\X) $ \vspace{0.1cm}
\item $\norm{\hslash_k -\hslash_{k-1}}_{\CC(\X)} \le \epsilon/5$ \vspace{0.1cm}
\item $\norm{\nabla \hslash_k -\nabla \hslash_{k-1}}_{\mathscr L^p (\X)} \le \epsilon/5$ \vspace{0.1cm}
\item $\norm{\nabla \hslash_1}_{\mathscr L^p (\X)}  \le \norm{\nabla \hslash_{0}}_{\mathscr L^p (\X)} - 2 \delta$, for some $\delta >0$; \\
  $\norm{\nabla \hslash_k}_{\mathscr L^p (\X)}  \le \norm{\nabla \hslash_{k-1}}_{\mathscr L^p (\X)}$, for $k=2,4$; \\
  $\norm{\nabla \hslash_k}_{\mathscr L^p (\X)} \le \norm{\nabla \hslash_{k-1}}_{\mathscr L^p (\X)} + \delta $, for $k=3,5$ \vspace{0.1cm}
\item  If $h_{k-1}$ is  a $\CC^\infty$-diffeomorphism outside of a compact subset of $\X$, then there exists a compact subset in $\X$ outside which we have  $\hslash_k \equiv \hslash_{k-1}$ for all $\epsilon >0$.
\end{enumerate}

\subsection{Partition of $\X$ into  cells}
Let us distinguish one particular Whitney type partition of $\Y$ and keep it fixed for the rest of our arguments.
\[\Y= \bigcup_{\nu =1}^\infty \overline{\Y_\nu}, \]
 where $\Y_\nu$ are mutually disjoint open dyadic squares such that
 \[\diam \Y_\nu \le \dist (\Y_\nu , \partial \Y) \le 3 \diam \Y_\nu \quad \mbox{for } \nu =1,2, \dots\]
unless $\Y=\R^2$, in which case $\Y_\nu$ are unit squares.
Thus the cover of $\Y$ by $\overline{\Y_\nu}$ is locally finite. The preimages
\[\X_\nu = h^{-1} (\Y_\nu), \qquad \nu =1,2, \dots\]
are Jordan domains which we call {\it  cells} in $\X$.
In the forthcoming Step~1 we shall need to further divide each cell into a finite number of {\it daughter cells} in $\X$. Note that  all but finite number of cells
$\X_\nu$, $\nu=1,2,...$ lie outside a given compact subset of $\X$.

\section*{Step 1}
To avoid undue indexing  in the forthcoming division of cells, we shall argue in two substeps.

\subsection*{Step 1a.} Examine one of the  cells in $\X$, say $\mathfrak X = \X_\nu$, for some fixed  $\nu=1,2, \dots$. Call it a {\it parent cell}. Thus $h(\mathfrak X) = \Upsilon$ is the corresponding Whitney square $\Upsilon= \Y_\nu \subset \Y$. To every $n=1,2, \dots$, there corresponds a partition of $\Upsilon$ into $4^n$-dyadic congruent  squares $\Upsilon_i$, $i=1, \dots, 4^n$
\[\overline{\Upsilon} = \overline{\Upsilon_1} \cup \dots \cup \overline{\Upsilon_{4^n}}.  \]
This gives rise to a division of $\mathfrak X$ into  daughter  cells $\mathfrak X_i = h^{-1} (\Upsilon_i)$
\[\overline{\mathfrak X} = \overline{\mathfrak X_1} \cup \overline{\mathfrak X_2} \cup \dots \cup \overline{\mathfrak X_{4^n}}.\]
We look at the homeomorphisms
\[h \colon \overline{\mathfrak X_i} \onto \overline{\Upsilon_i}, \qquad i=1,2, \dots 4^n\]
By virtue of Proposition~\ref{proreplace} we may replace them with $p$-harmonic homeomorphisms
\[\widetilde h_i = {\bf R}_{\mathfrak X_i} h \colon  \overline{\mathfrak X_i}  \onto  \overline{\Upsilon_i}, \qquad i=1,2, \dots , 4^n  \]
which coincide with $h$ on $\partial  {\mathfrak X_i} $. This procedure may not be necessary if $h \colon \mathfrak X_i \to \Upsilon_i$ is already a $\CC^\infty$-diffeomorphism.
In such cases we always use the \emph{trivial replacement} $\widetilde h_i=h$.
After all such replacements are made, we arrive at a homeomorphism
\[\widetilde h \colon  \overline{\mathfrak X} \onto  \overline{\Upsilon}  \]
which is a $\CC^\infty$-diffeomorphism in each cell $\mathfrak X_i$ and coincides with $h$ on $\partial \mathfrak X_i$. Obviously,
\[\widetilde h=h+ \sum_{i=1}^{4^n} [\widetilde h_i-h]_\circ \in h+ \mathcal A_\circ (\mathfrak X)\]
where  $ [\widetilde h_i -h]_\circ $ stands for zero extension of $\widetilde h_i-h$ outside $\mathfrak X_i$ and, therefore, belongs to $\mathcal A_\circ (\mathfrak X_i)$. Furthermore, by principle of minimal $p$-harmonic energy, we have
\[\EE_{\mathfrak X} [\widetilde h] =   \sum_{i=1}^{4^n} \EE_{\mathfrak X_i} [\widetilde h_i] \le   \sum_{i=1}^{4^n} \EE_{\mathfrak X_i} [{h}] = \EE_{\mathfrak X} [{h}] . \]
The eventual aim is to fix the number of daughter cells in $\mathfrak X$. For this we vary $n$ and look closely at the resulting homeomorphisms,
denoted by $f_n$. This sequence of mappings  is bounded in $\mathcal A (\mathfrak X)$. It actually converges to $h$ uniformly on $\overline{\mathfrak X}$. Indeed, given any point $x\in \overline{\mathfrak X}$, say $x\in \overline{\mathfrak X_i}$, for some $i=1,2, \dots, 4^n$, we have
\[\abs{f_n(x)-h(x)  } =   \abs{\widetilde h_i(x)-h(x)  } \le \diam \Upsilon_i = 2^{-n} \diam \Upsilon.\]
Thus
\[\lim_{n \to \infty}  f_n =h, \quad \mbox{ uniformly in } \overline{\mathfrak X}.\]
 On the other hand the mappings $f_n$ are bounded in the Sobolev space $\W^{1,p} (\mathfrak X)$, so converge to $h$ weakly in $\W^{1,p} (\mathfrak X)$. The key observation now is that
 \[ \norm{\nabla h}_{\mathscr L^p (\mathfrak X)}  \le \liminf_{n \to \infty} \norm{\nabla f_n}_{\mathscr L^p (\mathfrak X)}\le \norm{\nabla h}_{\mathscr L^p (\mathfrak X)} \]
 because of convexity of the energy functional. This gives
 \[\lim_{n \to \infty} \norm{\nabla f_n}_{\mathscr L^p (\mathfrak X)}   =  \norm{\nabla h}_{\mathscr L^p (\mathfrak X)} \]
 Then, the usual application of Clarkson's inequalities in $\mathscr L^p$-spaces, $1<p<\infty$, yields
 \[\lim_{n \to \infty} \norm{\nabla f_n-\nabla h}_{\mathscr L^p (\mathfrak X)} =0 \]
 meaning that $f_n -h \to 0$ in the norm topology of $\mathcal A(\mathfrak X)$. We can now determine the number $n=n_\nu=n(\mathfrak X)$, simply requiring the division of $\mathfrak X$  be fine enough to satisfy two conditions.
 \begin{equation}\label{eqstar}
 \begin{cases}
 \diam \Upsilon_i = 2^{-n} \diam \Upsilon \le \epsilon /5, \quad i=1, \dots, 4^n\\
 \norm{\nabla f_n-\nabla h}_{\mathscr L^p (\mathfrak X)}  \le \frac{\epsilon}{5 \cdot 2^\nu}
 \end{cases}
 \end{equation}
 where we recall that $\mathfrak X$ stands for $\X_\nu$.

\subsection*{Step 1b.} Now, having $n=n_\nu$ fixed for each  cell $\mathfrak X_\nu$, we   construct our first approximating mapping
 \[\hslash_1 \colon \X \onto \Y\]
 by setting
\[\hslash_1 := h + \sum_{\nu =1}^\infty [f_{n_\nu}-h]_\circ \in  h + \mathcal A_\circ (\mathbb X)\]
where, as always, $[f_{n_\nu}-h]_\circ$ stands for the zero extension of $f_{n_\nu}-h$ outside $\X_\nu$. This mapping is a $\CC^\infty$-diffeomorphism in every daughter cell.  Clearly, we have the condition
\begin{equation}
\hslash_1 -h \in \mathcal A_\circ (\X). \tag{$A_1$}
\end{equation}
Moreover, by the condition in~\eqref{eqstar} imposed on every $n_\nu$,
\begin{equation}
\norm{\hslash_1 -h }_{\CC(\X)} \le \sup_{\nu=1,2, \dots} \{\diam \Upsilon_i \colon \Upsilon_i \subset \Y_\nu, i=1, \dots , 4^{n_\nu}  \} \le \frac{\epsilon}{5} \tag{$B_1$}
\end{equation}
and
\begin{equation}
\norm{\nabla  \hslash_1- \nabla h}^p_{\mathscr L^p (\X)} = \sum_{\nu=1}^\infty \norm{\nabla  \hslash_1- \nabla h}^p_{\mathscr L^p (\X_\nu)}   \le  \left(\frac{\epsilon}{5}\right)^p \sum_{\nu=1}^\infty  \frac{1}{2^{\nu p}} <  \left(\frac{\epsilon}{5}\right)^p. \tag{$C_1$}
\end{equation}
Regarding condition $(D_1)$, we observe that summing up the energies over all daughter cells $\mathfrak X_i \subset \X_\nu$, $i=1,2, \dots 4^{n_\nu}$ and $\nu =1,2, \dots$, gives the total energy of $\hslash_1$ not larger than that of $h$. Even more, since $h$ fails to be a $\CC^\infty$-diffeomorphism in at least one of these cells,
the $p$-harmonic replacement takes place in this cell and, consequently, $\hslash_1$ has strictly smaller energy. Hence
\begin{equation}
\norm{\nabla \hslash_1}_{\mathscr L^p (\X)}    \le \norm{\nabla h}_{\mathscr L^p (\X)}   - 2 \delta, \quad \mbox{ for some } \delta >0.\tag{$D_1$}
\end{equation}
Regarding condition $(E_1)$, we note that under the assumption therein we made only a finite number of nontrivial ($p$-harmonic) replacements.
The same remark will apply to the subsequent steps and will not be mentioned again. The step 1 is complete.


 Before proceeding to Step 2, let us put all daughter cells in $\X$  in a single sequence
 \[\mathfrak X^1 , \mathfrak X^2 , \dots \subset \X .\]
 Thus from now on the daughter cells from different parents are indistinguishable as far as the mapping $\hslash_1$ is concerned. The point is that $\hslash_1$ is a $\CC^\infty$-diffeomorphism in every such cell, a property that will be pertinent to all new cells coming later either by splitting or merging the existing cells.  Note that the images $\Upsilon^\alpha = h(\mathfrak X^\alpha)$, $\alpha = 1,2, \dots$, form a partition of $\Y$ into dyadic squares
 \[\Y=\bigcup_{\alpha =1}^\infty  \overline{\Upsilon^\alpha}, \qquad \mbox{ where }\quad \diam \Upsilon^\alpha \le \frac{\epsilon}{5}. \]

\begin{center}
\begin{figure}[h]
\includegraphics[width=0.9\textwidth]{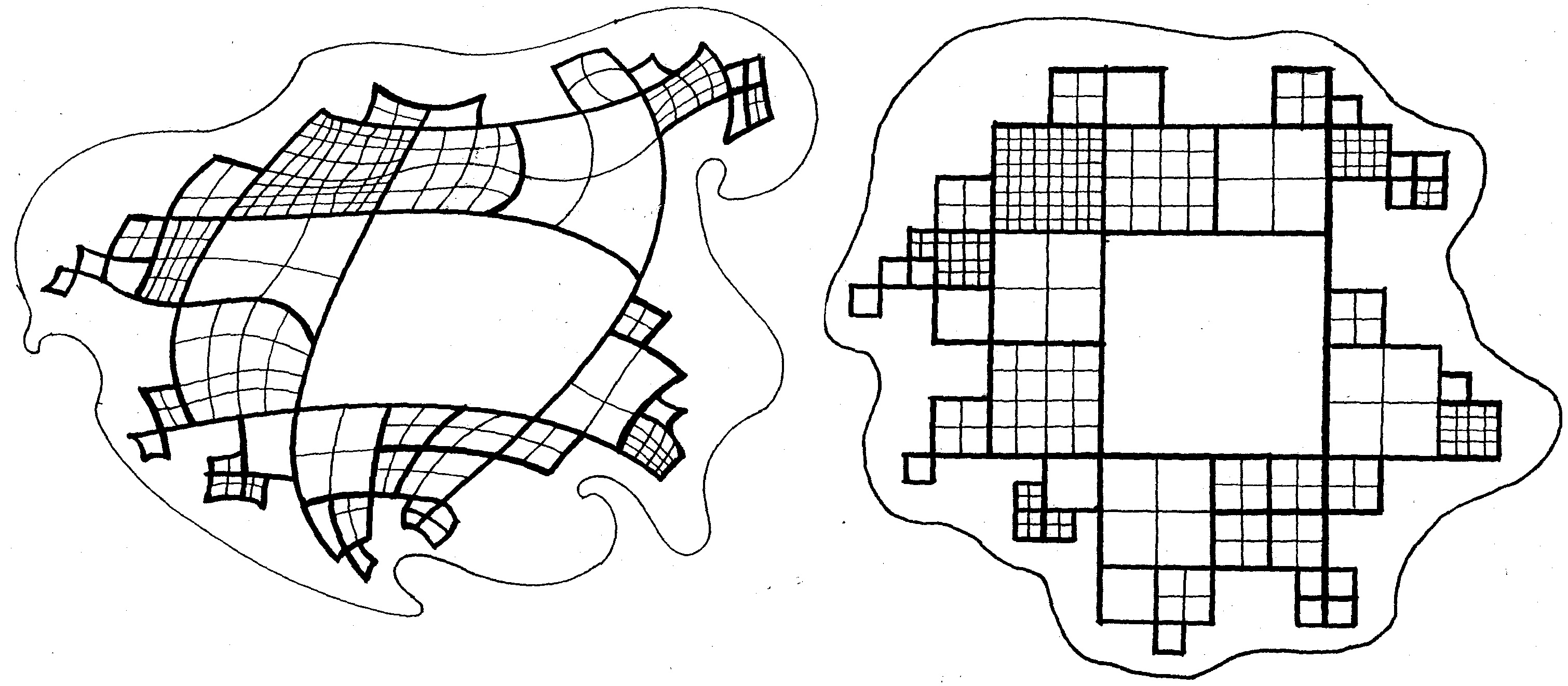}
\caption{ $\hslash_1$ is a $\CC^\infty$-diffeomorphism in each cell $\mathfrak X^\alpha \subset \X$.}
\end{figure}
\end{center}

\section*{Step 2}

\subsection*{Step 2a.} (Adjacent cells) Let $\mathcal C(\Y) \subset \Y$ be the collection of all corners of dyadic squares $\Upsilon^\alpha$, $\alpha =1,2, \dots$, and $\V (\X) \subset \X$ denote the set of their preimages under $h$, called {\it vertices of  cells}. Whenever two closed cells $\overline{\mathfrak X^\alpha}$ and $\overline{\mathfrak X^\beta}$, $\alpha \ne \beta$, intersect, their common part is either a point in $\V (\X)$ or an edge, that is, a closed Jordan arc with endpoints in $\V(\X)$. In this latter case we say that ${\mathfrak X^\alpha}$ and ${\mathfrak X^\beta}$ are adjacent cells with common edge
\[\overline{C^{\alpha \, \beta}} = \overline{\mathfrak X^\alpha} \cap \overline{\mathfrak X^\beta}.\]
This is the closure of a Jordan open arc
${C^{\alpha \, \beta}} = \overline{C^{\alpha \, \beta}} \setminus \V (\X).$
The mappings
\[\hslash_1 \colon \mathfrak X^\alpha \onto \Upsilon^\alpha \quad \mbox{and} \quad  \hslash_1 \colon \mathfrak X^\beta \onto \Upsilon^\beta \]
are $\CC^\infty$-diffeomorphisms but they do not necessarily match smoothly along the edges. We shall now produce a new cell $\mathfrak X^{\alpha \, \beta}$, a daughter of the adjacent cells $\mathfrak X^\alpha$ and $\mathfrak X^\beta$, such that
\[C^{\alpha \, \beta} \subset \mathfrak X^{\alpha \, \beta} \subset  \mathfrak X^{\alpha} \cup  C^{\alpha \, \beta} \cup  \mathfrak X^{\beta} .\]
To construct $ \mathfrak X^{\alpha \, \beta}$ we look at the adjacent dyadic squares $\overline{\Upsilon^\alpha}$ and $\overline{\Upsilon^\beta}$ in $\Y$. The intersection $\overline{\Upsilon^\alpha} \cap \overline{\Upsilon^\beta}=h(\overline{C^{\alpha \, \beta}}) $ is a closed interval. Let $R$ be a number greater than the length of  $h({C^{\alpha \, \beta}})$ to be chosen sufficiently large later on. There exist exactly two open disks of radius $R$ for which  $h({C^{\alpha \, \beta}})$ is a chord. Their intersection, denoted by $\mathcal L^{\alpha \, \beta}$, is a symmetric doubly convex lens of curvature $R^{-1}$. Thus $\mathcal L^{\alpha \, \beta}$ is enclosed between two open circular arcs $\gamma^{\alpha \, \beta} = \Upsilon^\alpha \cap \partial \mathcal L^{\alpha \, \beta} \subset \Upsilon^\alpha$ and  $\gamma^{\beta \, \alpha} = \Upsilon^\beta \cap \partial \mathcal L^{\alpha \, \beta} \subset \Upsilon^\beta$. Note that $ \mathcal L^{\alpha \, \beta} =  \mathcal L^{\beta \, \alpha}$, but $\gamma^{\alpha \, \beta} \ne \gamma^{\beta \, \alpha}$. We call
\begin{equation}\label{eq3stars}
\mathfrak X^{\alpha \, \beta}= \hslash_1^{-1} (\mathcal L^{\alpha \, \beta}), \quad \mbox{a daughter of the adjacent cells } \mathfrak X^\alpha \mbox{ and } \mathfrak X^\beta.
\end{equation}
As the curvature of the lens $ \mathcal L^{\alpha \, \beta}$ approaches zero,  the area of $ \mathfrak X^{\alpha \, \beta} $ tends to $0$. This allows us to choose $R$ so that
\begin{equation}\label{eq2stars}
\norm{\nabla \hslash_1}_{\mathscr L^p (\mathfrak X^{\alpha \, \beta})} \le \frac{\epsilon}{5 \cdot 2^{\alpha + \beta}}.
\end{equation}
The lenses $ \mathcal L^{\alpha \, \beta} $ are disjoint because the opening angle of each lens (the angle between arcs at their common endpoints) is at most $\pi/3$ and their long axes are either parallel or orthogonal, see Figure~\ref{lens}. Therefore, the  cells $\mathfrak X^{\alpha \, \beta}= \hslash_1^{-1} ( \mathcal L^{\alpha \, \beta} )$  are also disjoint. However, their closures may have a common point that lies in $\V (\X)$. The boundary of $\mathfrak X^{\alpha \, \beta}$ consists of two open arcs
\[\Gamma^{\alpha \, \beta} = \mathfrak X^\alpha \cap \partial \mathfrak X^{\alpha \, \beta} \quad \mbox{and} \quad \Gamma^{\beta \, \alpha} = \mathfrak X^\beta \cap \partial \mathfrak X^{\alpha \, \beta} \]
plus their endpoints. These open arcs are $\CC^\infty$-smooth because they come as images of the  circular arcs enclosing the lens $\mathcal L^{\alpha \, \beta}$ under   a $\CC^\infty$-diffeomorphism.

\begin{center}
\begin{figure}[h]
\includegraphics[width=0.9\textwidth]{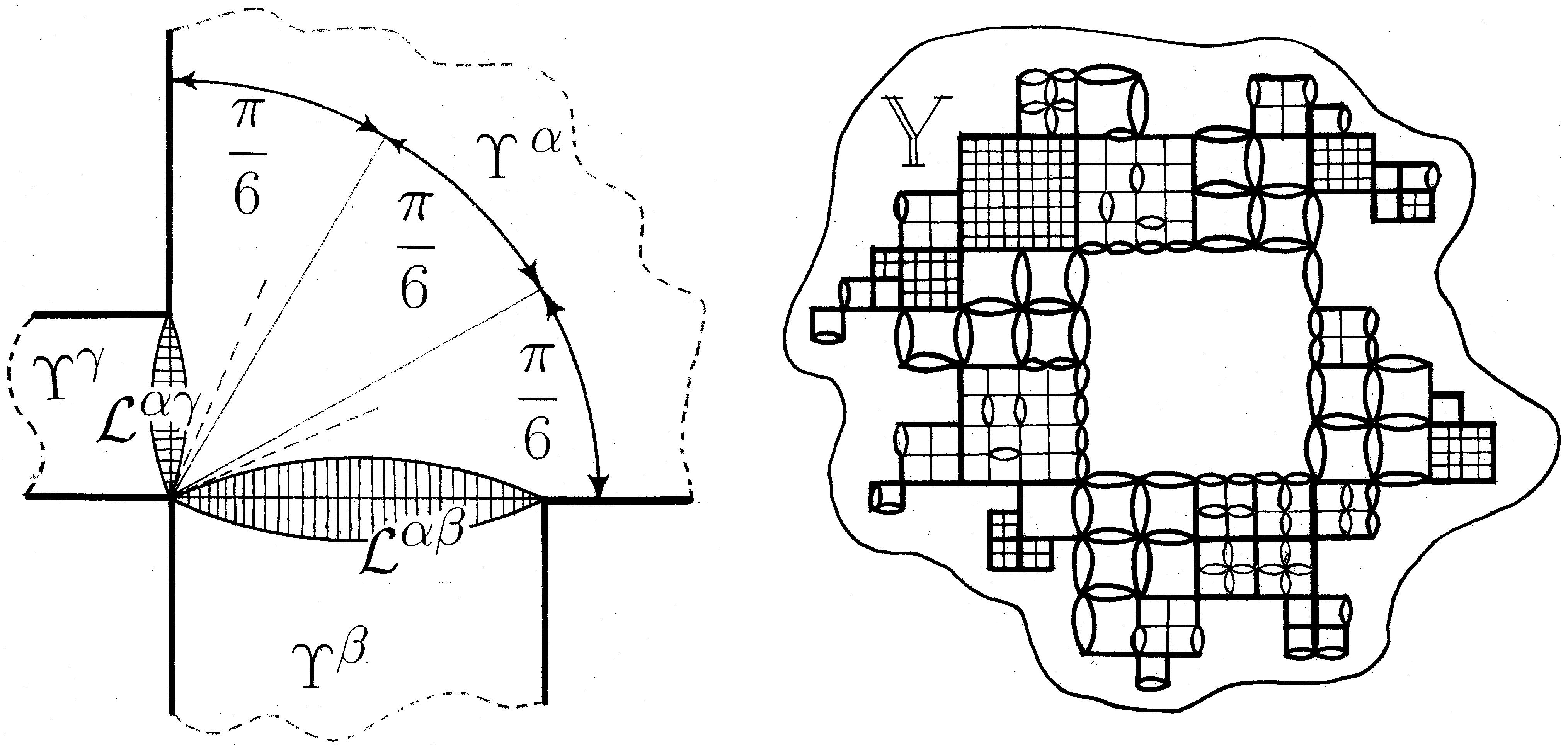}
\caption{Lenses.}\label{lens}
\end{figure}
\end{center}

\begin{remark}
In what follows we shall consider only the pairs $(\alpha, \beta)$ of indices $\alpha =1,2, \dots$ and $\beta =1,2, \dots$ which correspond to adjacent cells. Such pairs will be designated the symbol $\alpha \beta$.
\end{remark}

\subsection*{Step 2b.} (Replacements in $\mathfrak X^{\alpha \, \beta}$) The lenses $\mathcal L^{\alpha \, \beta} \subset \Y$ are convex, so with the aid of Proposition~\ref{proreplace} and Theorem~\ref{RKCthm}, we may replace $\hslash_1 \colon \mathfrak X^{\alpha \, \beta} \to \mathcal L^{\alpha \, \beta}$ with the $p$-harmonic extension of $\hslash_1 \colon \partial \mathfrak X^{\alpha \, \beta} \to \partial \mathcal L^{\alpha \, \beta}$. We do this, and denote the result by $\hslash_2^{\alpha \beta} \colon \mathfrak X^{\alpha \beta} \to \mathcal L^{\alpha \beta}$, only on the cells in which $\hslash_1\colon \mathfrak X^\alpha\cup \mathfrak X^\beta \cup \mathfrak X^{\alpha \beta}\to \R^2$ is not a $\CC^\infty$-diffeomorphism. In other cells we set $\hslash_2^{\alpha \beta} = \hslash_1$.  In either case $\hslash_2^{\alpha \beta}  \in \hslash_1 + \mathcal A_\circ (\mathfrak X^{\alpha \, \beta})$ so we define
\[\hslash_2=\hslash_1 + \sum_{\alpha \beta} [\hslash^{\alpha \, \beta}_2 - \hslash_1]_\circ .  \]
Thus we have
\begin{equation}
\hslash_2-\hslash_1 \in \mathcal A_\circ (\X). \tag{$A_2$}
\end{equation}
The advantage of using  $\hslash_2$ in the next step lies in the fact that it  is not only a $\CC^\infty$-diffeomorphism in every cell, but also  is $\CC^\infty$-smooth with positive Jacobian determinant, up to each edge of the cells created here. These edges are $\CC^\infty$-smooth open arcs.  By cells created here we mean not only $\mathfrak X^{\alpha \, \beta}$ but also those obtained from the parent cell $\mathfrak X^\alpha$ by removing the adjacent daughters; that is,
\[\mathfrak X^\alpha \setminus \bigcup_{\alpha \beta} \mathfrak X^{\alpha \, \beta}, \qquad \alpha =1,2, \dots\]
See Figure~\ref{cells}. The estimates of $\hslash_2$ run as follows. By~\eqref{eqstar} we have,
\begin{equation}
\norm{\hslash_2- \hslash_1}_{\CC (\X)} \le \sup\limits_{\alpha  \beta} \{\diam \mathcal L^{\alpha \, \beta}\} \le \sup\limits_{\alpha} \{ \diam \Y^\alpha \} \le \frac{\epsilon}{5}. \tag{$B_2$}
\end{equation}In view of the minimum $p$-harmonic energy principle, we have
\[
\begin{split}
\norm{\nabla \hslash_2 - \nabla \hslash_1}_{\mathcal L^p (\X)} & = \sum_{\alpha  \beta } \norm{\nabla \hslash_2 - \nabla \hslash_1}_{\mathcal L^p (\cup \mathfrak X^{\alpha \, \beta})} \\
& \le  \sum_{\alpha  \beta } \left[ \norm{\nabla \hslash_2}_{\mathcal L^p ( \mathfrak X^{\alpha \, \beta})} + \norm{\nabla \hslash_1}_{\mathcal L^p ( \mathfrak X^{\alpha \, \beta})}      \right] \\
& \le 2   \sum_{\alpha  \beta }  \norm{\nabla \hslash_1}_{\mathcal L^p ( \mathfrak X^{\alpha \,  \beta})}  \le  \frac{2\epsilon}{5}   \sum_{\alpha  \beta }  2^{-\alpha-\beta }.
\end{split}
\]
by~\eqref{eq2stars}. Hence
\begin{equation}
\norm{\nabla  \hslash_2 - \nabla \hslash_1}_{\mathcal L^p (\X)} \le \frac{\epsilon}{5}. \tag{$C_2$}
\end{equation}
The minimum energy principle also yields estimate
\[
\begin{split}
\norm{\nabla  \hslash_2}^p_{\mathcal L^p (\X)} & = \norm{\nabla \hslash_2}^p_{\mathcal L^p (\cup \mathfrak X^{\alpha \, \beta})} + \norm{\nabla \hslash_1}^p_{\mathcal L^p (\X \setminus \cup \mathfrak X^{\alpha \, \beta})} \\
& \le  \norm{\nabla \hslash_1}^p_{\mathcal L^p (\cup \mathfrak X^{\alpha \, \beta})} + \norm{\nabla \hslash_1}^p_{\mathcal L^p (\X \setminus \cup \mathfrak X^{\alpha \, \beta})} =   \norm{\nabla \hslash_1}^p_{\mathcal L^p (\X)}.
\end{split}
\]
In particular
\begin{equation}
 \norm{\nabla \hslash_2}_{\mathcal L^p (\X)} \le  \norm{\nabla \hslash_1}_{\mathcal L^p (\X)}, \tag{$D_2$}
\end{equation}
completing the proof of Step 2.

\begin{center}
\begin{figure}[h]
\includegraphics[width=0.7\textwidth]{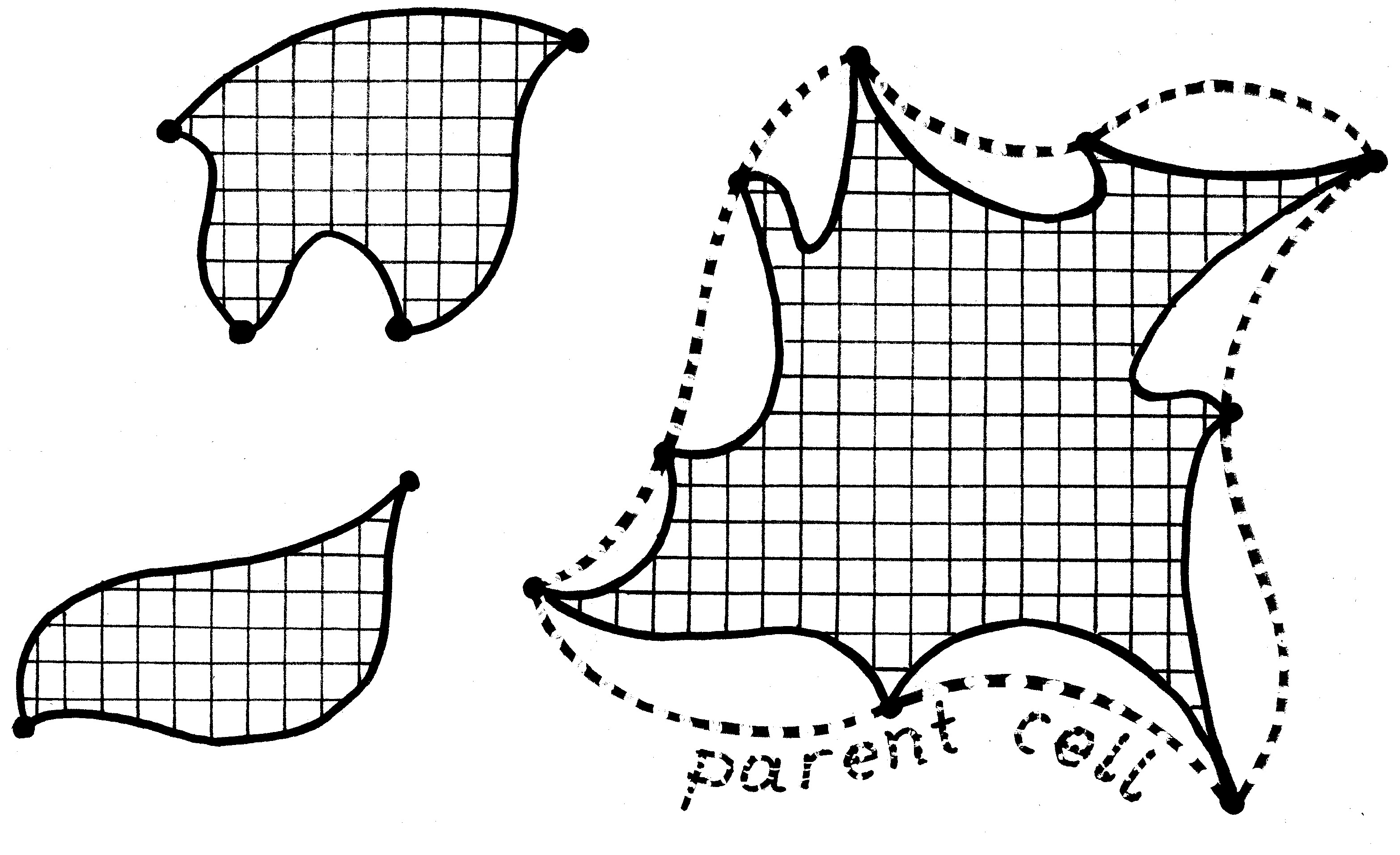}
\caption{Three types of cells.}\label{cells}
\end{figure}
\end{center}

Note that $\hslash_2$ is locally bi-Lipschitz in $\X\setminus \V (\X)$. The exceptional set $\V (\X)$ is discrete.

\section*{Step 3}

We shall now merge all the adjacent cells together, by smoothing $\hslash_2$ around the edges $\Gamma^{\alpha \, \beta} \subset \mathfrak X^\alpha$. To achieve proper estimates we need to remove small neighborhoods of all vertices, outside which $\hslash_2$ is certainly locally bi-Lipschitz.

\subsection*{Step 3a.}  First we cover the set $\mathcal C(\Y)$ of corners of dyadic squares by disks $\mathbb D_c$ centered at $c\in \mathcal C (\Y)$.
These disks will be chosen small enough to satisfy all the conditions listed   below.
\begin{enumerate}[(i)]
\item $\diam \mathbb D_c < \epsilon / 5$ for every $c\in \mathcal C(\Y)$,\vskip0.15cm
\item $\displaystyle \sum_{v\in \V(\X)} \int_{\mathbb F_v}  \abs{\nabla \hslash_2}^p \le \bigg(\frac{\epsilon}{20}\bigg)^p$, where $\mathbb F_v=\hslash_2^{-1}(\mathbb D_c)$, $c=\hslash_2(v)=h(v)$.
\end{enumerate}

Denote by $\X_\circ = \X \setminus \bigcup \overline{\mathbb F_v}$. We truncate each edge $\Gamma^{\alpha \, \beta}$ near the endpoints by setting
\begin{equation}
\Gamma_\circ^{\alpha \, \beta} = \Gamma^{\alpha \, \beta} \cap \X_\circ.
\end{equation}
These are mutually disjoint open arcs; their closures are isolated continua in $\X \setminus \V (\X)$. This means that there are disjoint neighborhoods of them. We are actually interested in neighborhoods $\mathbb U^{\alpha \, \beta} \subset \mathfrak X^\alpha$ of $\Gamma_\circ^{\alpha \, \beta}$ that are Jordan domains in which $\Gamma_\circ^{\alpha \, \beta} \subset \mathbb U^{\alpha \, \beta}$ are $\CC^\infty$-smooth crosscuts with two endpoints in $\partial \mathbb U^{\alpha \, \beta}$, see Section~\ref{prel}. It is geometrically clear that such mutually disjoint neighborhoods exist. Now the stage for next substep is established.

\subsection*{Step 3b.} ($\CC^\infty$-replacement within $\mathbb U^{\alpha \, \beta}$) It is at this stage that we will improve $\hslash_2$ in $\mathbb U^{\alpha \, \beta}$ to a $\CC^\infty$-smooth diffeomorphism with no harm to the previously established estimates for $\hslash_2$. The tool is Proposition~\ref{prop41}. As always, we shall make no replacement of $\hslash_2 \colon \mathbb U^{\alpha \beta} \to \Upsilon^\alpha$ if it is already $\CC^\infty$-diffeomorphism. Recall that we have a bi-Lipschitz mapping $\hslash_2 \colon \mathbb U^{\alpha \, \beta} \to \hslash_2 (\mathfrak X^\alpha)= \Upsilon^\alpha$ that takes the crosscut $\Gamma_\circ^{\alpha \, \beta} \subset \mathbb U^{\alpha \, \beta}$ onto a circular arc. Denote the components $\mathbb U_+^{\alpha \beta} = \mathbb U^{\alpha \beta} \setminus \overline{\mathfrak X^{\alpha \, \beta}}$ and  $\mathbb U_-^{\alpha \beta} = \mathbb U^{\alpha \beta} \cap {\mathfrak X^{\alpha \, \beta}}$. Furthermore, we have
\[\abs{D \hslash_2} \le m_{\alpha \beta} \quad \mbox{ and } \quad \det D \hslash_2 \ge \frac{1}{m_{\alpha \beta}},\quad \text{for some }\ m_{\alpha\beta}>0\]
on each component. The mappings $\hslash_2 \colon \mathbb U_+^{\alpha \beta} \to \Upsilon^\alpha $ and $\hslash_2 \colon \mathbb U_-^{\alpha \beta} \to \Upsilon^\alpha $ are $\CC^\infty$-diffeomorphisms up to $\Gamma^{\alpha \, \beta}_\circ$. In accordance with Proposition~\ref{prop41} we find a constant $M_{\alpha \beta}$ such that: whenever open set $\mathbb V^{\alpha \, \beta} \subset \mathbb U^{\alpha \, \beta}$ contains the crosscut $\Gamma_\circ^{\alpha \, \beta}$ there exists a homeomorphism $\hslash_3^{\alpha \, \beta} \colon \overline{\mathbb U^{\alpha \, \beta}} \onto \hslash_2 (\overline{\mathbb U^{\alpha \, \beta}})$ which is a $\CC^\infty$-diffeomorphism in $\mathbb U^{\alpha \, \beta}$, with the following properties
\begin{itemize}
\item $\hslash_3^{\alpha \, \beta}  \equiv \hslash_2 $ on $(\overline{\mathbb U^{\alpha \, \beta}} \setminus \mathbb V^{\alpha \, \beta})\cup \Gamma_\circ^{\alpha \, \beta}$;
\item $\abs{\nabla \hslash_3^{\alpha \, \beta} } \le M_{\alpha \beta}$  and $\det \nabla \hslash_3^{\alpha \, \beta} \ge \frac{1}{M_{\alpha \beta}}$ in $\mathbb U^{\alpha \, \beta}$.
\end{itemize}
Since $M_{\alpha \beta}$ does not depend on $\mathbb V^{\alpha \, \beta}$ it will be advantageous to take neighborhoods $\mathbb V^{\alpha \, \beta}$ of $\Gamma_\circ^{\alpha \, \beta}$ thin enough to satisfy
\begin{itemize}
\item $\overline{\mathbb V^{\alpha \, \beta}} \subset \mathbb U^{\alpha \, \beta} \cup \overline{\Gamma_\circ^{\alpha \, \beta}  }$;
\item $\abs{\mathbb V^{\alpha \, \beta}} \le \frac{1}{5^p \cdot 2^{\alpha + \beta}} \left[ \frac{\epsilon}{m_{\alpha \beta} + M_{\alpha \beta} }\right]^p$ and also $\abs{\mathbb V^{\alpha \, \beta}} \le \frac{\delta}{2^{\alpha + \beta} M_{\alpha \beta}}$.
\end{itemize}
Note that   $\hslash_3^{\alpha \, \beta}, \hslash_2 \in \W^{1, \infty} (\mathbb U^{\alpha \, \beta}) \subset  \W^{1, p} (\mathbb U^{\alpha \, \beta})$ and   $\hslash_3^{\alpha \, \beta}= \hslash_2 $ on $\partial \mathbb U^{\alpha \, \beta}$, so we have
\[\hslash_3^{\alpha \, \beta} - \hslash_2 \in \W_\circ^{1,p} (\mathbb U^{\alpha \, \beta}).\]

\subsection*{Step 3c.} We now define a homeomorphism $\hslash_3 \colon \X \onto \Y$ by the rule
\[
\hslash_3 = \begin{cases} \hslash_3^{\alpha \, \beta} & \quad \mbox{in } \mathbb U^{\alpha \, \beta}\\
\hslash_2 & \quad \mbox{in } \X \setminus \bigcup\limits_{\alpha \beta} \mathbb U^{\alpha \, \beta}.
\end{cases}
\]
Obviously, $\hslash_3$ is a $\CC^\infty$-diffeomorphism in $\X_\circ$ and $\hslash_3 - \hslash_2 \in \W_\circ^{1,p} (\mathbb X_\circ)$. Since $\hslash_3$ coincides with $\hslash_2$ outside $\X_\circ$ we have $\hslash_3 = \hslash_2 + [\hslash_3 - \hslash_2]_\circ$. Hence
\begin{equation}
\hslash_3 - \hslash_2 \in \mathcal A_\circ (\X). \tag{$A_3$}
\end{equation}
Then, for every $x\in \X$,
\[\abs{\hslash_3(x) - \hslash_2 (x)} \le \begin{cases}\diam \hslash_2 ( \mathbb U^{\alpha \, \beta}), \; & \mbox{for } x\in  \mathbb U^{\alpha \, \beta}\\
0, & \mbox{otherwise} \end{cases} \; \le \diam \Upsilon^\alpha \le \frac{\epsilon}{5} \]
meaning that
\begin{equation}
\norm{\hslash_3 - \hslash_2 }_{\CC (\X)} \le \frac{\epsilon}{5}. \tag{$B_3$}
\end{equation}
The computation of $p$-norms goes as follows
\[
\begin{split}
\norm{\nabla \hslash_3 - \nabla \hslash_2 }^p_{\mathscr L^p (X)} &= \sum_{\alpha \beta} \int_{\mathbb V^{\alpha \, \beta}} \abs{\nabla \hslash_3 - \nabla \hslash_2}^p \\
& \le \sum_{\alpha  \beta} \abs{  \mathbb V^{\alpha \, \beta} } \left[  \norm{\nabla \hslash_3}_{\CC (  \mathbb V^{\alpha \, \beta} )}  +  \norm{\nabla \hslash_2}_{\CC (  \mathbb V^{\alpha \, \beta} )}   \right]^p \\
& \le \sum_{\alpha  \beta} \abs{  \mathbb V^{\alpha \, \beta} } \left(m_{\alpha \beta}   + M_{\alpha \beta} \right)^p \le  \sum_{\alpha \beta} \frac{\epsilon^p}{5^p\, 2^{\alpha + \beta}} \le \left(\frac{\epsilon}{5}\right)^p .
\end{split}
\]
Hence
\begin{equation}
\norm{\nabla \hslash_3 - \nabla \hslash_2}_{\mathscr L^p (X)} \le \frac{\epsilon}{5}. \tag{$C_3$}
\end{equation}
In the finite energy case, when $\norm{\nabla \hslash_2}_{\mathscr L^p (\X)} < \infty $, we observe that
\[\norm{\nabla \hslash_3}_{\mathscr L^p (\X  \setminus \cup\mathbb V^{\alpha \, \beta}   )} = \norm{\nabla \hslash_2}_{\mathscr L^p (\X  \setminus \cup \mathbb V^{\alpha \, \beta}   )} \le \norm{\nabla \hslash_2}_{\mathscr L^p (\X)} .\]
Therefore, by triangle inequality,
\[
\begin{split}  \norm{\nabla \hslash_3}_{\mathscr L^p (\X)} & \le  \norm{\nabla \hslash_2}_{\mathscr L^p (\X)} + \sum_{\alpha  \beta}  \norm{\nabla \hslash_3}_{\mathscr L^p ( \mathbb V^{\alpha \, \beta}   )} \\
& \le  \norm{\nabla \hslash_2}_{\mathscr L^p (\X)} +  \sum_{\alpha  \beta} \abs{  \mathbb V^{\alpha \, \beta} } \cdot \norm{\nabla \hslash_3}_{\CC (\mathbb V^{\alpha \, \beta}) }\\
& \le  \norm{\nabla \hslash_2}_{\mathscr L^p (\X)} + \sum_{\alpha  \beta} \frac{\delta}{2^{\alpha + \beta} M_{\alpha \beta}} \cdot M_{\alpha \beta}
\end{split}
\]
which yields
\begin{equation}
 \norm{\nabla \hslash_3}_{\mathscr L^p (\X)} \le  \norm{\nabla \hslash_2}_{\mathscr L^p (\X)} + \delta. \tag{$D_3$}
\end{equation}
The third step is completed.

\section*{Step 4}

We have already upgraded the mapping $h$ to a homeomorphism $\hslash_3 \colon \X \onto \Y$ that is a $\CC^\infty$-diffeomorphism in $\X_\circ = \X \setminus \bigcup_{v\in \V (\X)} \overline{\mathbb F_v}$, where $\mathbb F_v$ are small surroundings of the vertices of cells. Their images $\hslash_3 (\mathbb F_v)=\hslash_2 (\mathbb F_v)= \mathbb D_c $ are small disks centered at $c=h(v)$. In Step~3a, one of the preconditions on those disks was that $\diam \mathbb D_c < \epsilon /5$. Furthermore, the closed disks $\overline{\mathbb D_c}$ are isolated continua in $\Y$ for all $c\in \mathcal C (\Y)$, so are the sets $\overline{\mathbb F_v}$ in $\X$. We shall now consider slightly larger concentric open disks $\mathbb D_c' \supset \overline{\mathbb D_c}$, $c \in \mathcal C (\Y)$, and their preimages $\mathbb F_v' = h_3^{-1} (\mathbb D_c') \subset \X$, $v=h^{-1}(c) \in \V (\X)$. The annulus $\mathbb D_c' \setminus \overline{\mathbb D_c}$ will be thin enough to ensure that $\mathbb D_c'$ are still disjoint,
\[\diam \mathbb D_c' < \frac{\epsilon}{5} \quad \mbox{ for all } c \in \mathcal C(\Y)\]
and
\[\sum_{v \in \V(\X)} \norm{\nabla \hslash_3}^p_{\mathcal L^p (\mathbb F_v' \setminus \mathbb F_v)} \le \left(\frac{\epsilon}{20}\right)^p.\]

\begin{center}
\begin{figure}[h]
\includegraphics[width=0.6\textwidth]{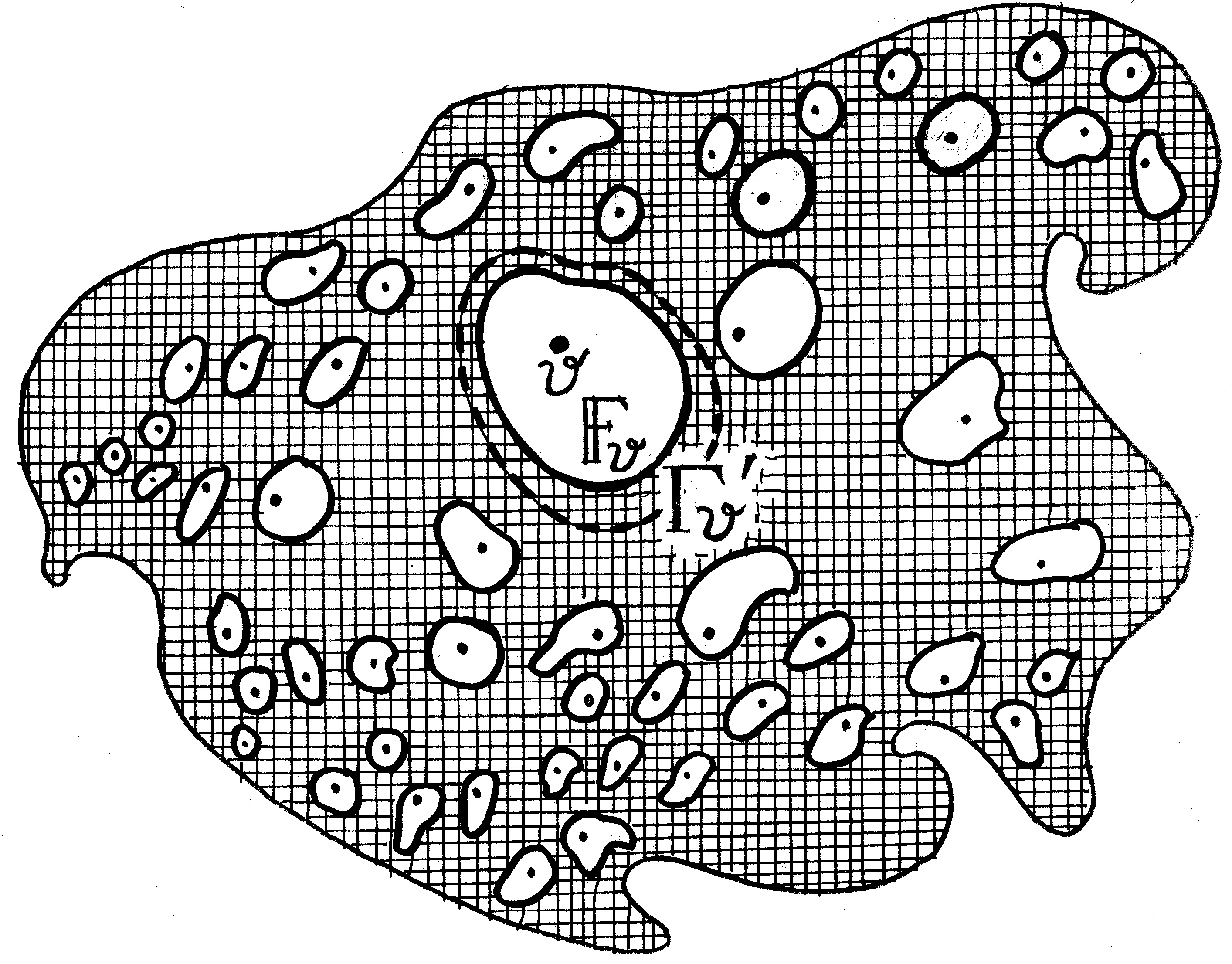}
\caption{Neighborhoods of vertices.}
\end{figure}
\end{center}

Let $\Gamma_v'$, $v\in \V(\X)$, denote the boundary of $\mathbb F_v'$. These are $\CC^\infty$-smooth Jordan curves. We now define a homeomorphism $\hslash_4 \colon \X \onto \Y$ by performing $p$-harmonic replacement of mappings $\hslash_3 \colon \mathbb F'_v \onto \mathbb D_c'$, whenever such a mapping fails to be $\CC^\infty$-diffeomorphism. Thus every  $\hslash_4 \colon \mathbb F'_v \onto \mathbb D_c'$ is a $\CC^\infty$-diffeomorphism up to $\Gamma'_v$. Moreover $\hslash_4 \in \hslash_3 + \W_{\circ}^{1,p} (\mathbb F'_c)$, so
\begin{equation}
\hslash_4 - \hslash_3 \in \mathcal A_\circ (\X). \tag{$A_4$}
\end{equation}
For every $x\in \X$, we have
\[
\abs{\hslash_4 (x)- \hslash_3 (x)} \le \begin{cases} \diam \mathbb D_c' \; & \mbox{ in } \mathbb F_v', \ c =h(v)\\
0 & \mbox{otherwise}
\end{cases} \ \le \frac{\epsilon}{5}.
\]
Hence
\begin{equation}
\norm{\hslash_4- \hslash_3}_{\CC (\X)} \le \frac{\epsilon}{5}. \tag{$B_4$}
\end{equation}
By virtue of the minimum energy principle we compute the $p$-norms
\[
\begin{split}
\norm{\hslash_4- \hslash_3 }^p_{\mathscr L^p (\X)} &= \sum_{v\in \V(\X)} \norm{\hslash_4- \hslash_3 }^p_{\mathscr L^p (\mathbb F_v')} \\
& \le  \sum_{v\in \V(\X)} \left[ \norm{\hslash_4}_{\mathscr L^p (\mathbb F_v')} + \norm{\hslash_3 }_{\mathscr L^p (\mathbb F_v')}   \right]^p \\
& \le 2^p  \sum_{v\in \V(\X)}  \norm{\hslash_3 }^p_{\mathscr L^p (\mathbb F_v')} \\
& \le 2^{2p-1}   \sum_{v\in \V(\X)} \left[  \norm{\hslash_3 }^p_{\mathscr L^p (\mathbb F_v' \setminus \mathbb F_v)} +     \norm{\hslash_3 }^p_{\mathscr L^p (\mathbb F_v)}  \right] \\
& \le 2^{2p-1}   \left[  \left(\frac{\epsilon}{20}\right)^p+      \sum_{v\in \V(\X)} \norm{\hslash_2 }^p_{\mathscr L^p (\mathbb F_v)}  \right] \\
& \le 2^{2p}  \left(\frac{\epsilon}{20}\right)^p = \left(\frac{\epsilon}{5}\right)^p.
\end{split}
\]
Hence
\begin{equation}
\norm{ \hslash_4- \hslash_3 }_{\mathscr L^p (\X)} \le \frac{\epsilon}{5}. \tag{$C_4$}
\end{equation}
Again by minimum energy principle we find that
\begin{equation}
\norm{\hslash_4}^p_{\mathscr L^p (\X)} \le \norm{\hslash_3}^p_{\mathscr L^p (\X)}. \tag{$D_4$}
\end{equation}
Just as in the previous steps, condition $(E_4)$ remains valid, finishing Step 4.

\section*{Step 5}

The final step consists of smoothing $\hslash_4$ in a neighborhood of each smooth Jordan curve $\Gamma_v'$, $v \in \V (\X)$. We argue in much the same way as in Step~3, but this time we appeal to Proposition~\ref{prop42} instead of Proposition~\ref{prop41}. By smoothing $\hslash_4$ in a sufficiently thin neighborhood of each $\Gamma_v'$ we obtain a $\CC^\infty$-diffeomorphism $\hslash_5 \colon \X \onto \Y$,
\begin{equation}
\hslash_5- \hslash_4 \in \mathcal A_\circ (\X). \tag{$A_5$}
\end{equation}
\begin{equation}
\norm{\hslash_5- \hslash_4}_{\CC (\X)} \le \frac{\epsilon}{5}. \tag{$B_5$}
\end{equation}
\begin{equation}
\norm{\hslash_5- \hslash_4}_{\mathscr L^p (\X)} \le \frac{\epsilon}{5}. \tag{$C_5$}
\end{equation}
\begin{equation}
\norm{\hslash_5}_{\mathscr L^p (\X)} \le \norm{\hslash_4}_{\mathscr L^p (\X)}+ \delta. \tag{$D_5$}
\qed
\end{equation}

\section{Open questions}
\begin{questionstar}
 Does Theorem~\ref{thmmain} extend to $n=3$?
 \end{questionstar}
 \begin{questionstar}
A bi-Sobolev homeomorphism $h \colon \X \onto \Y$ is a mapping of class $\W^{1,p}(\X , \Y)$, $1 \le p < \infty$, whose inverse $h^{-1}\colon  \Y \onto \X$  belongs to a Sobolev class  $\W^{1,q}(\Y , \X)$, $1 \le q < \infty$. Can $h$ be approximated by bi-Sobolev diffeomorphisms $\{h_\ell\}$ so that $h_\ell \to h$ in $\W^{1,p}(\X , \Y)$ and $h^{-1}_\ell \to h^{-1}$ in $\W^{1,q}(\Y , \X)$?
\end{questionstar}

\bibliographystyle{amsplain}

\end{document}